\documentclass[a4paper,11pt,reqno]{amsart}
\usepackage{amssymb,amsmath,hyperref}
\title{A general formula for Hecke-type false theta functions}

\author{Eric T. Mortenson}
\address{Department of Mathematics and Computer Science, Saint Petersburg State University, Saint Petersburg, 199034, Russia}
\email{etmortenson@gmail.com}

\renewcommand\theta{\vartheta}

\newcommand\sg{\operatorname{sg}}
\newtheorem{theorem}{Theorem}

\theoremstyle{definition}
\newtheorem{definition}[theorem]{Definition}

\numberwithin{theorem}{section} 
\numberwithin{equation}{section}

\makeatletter
\@namedef{subjclassname@2020}{\textup{2020} Mathematics Subject Classification}
\makeatother

\begin{document}

\date{27 December 2022}



\begin{abstract}
In recent work where Matsusaka generalizes the relationship between Habiro-type series and false theta functions after Hikami,  five families of Hecke-type double-sums of the form
\begin{equation*}
\left( \sum_{r,s\ge 0 }-\sum_{r,s<0}\right)(-1)^{r+s}x^ry^sq^{a\binom{r}{2}+brs+c\binom{s}{2}},
\end{equation*}
where $b^2-ac<0$, are decomposed into sums of products of theta functions and false theta functions.   Here we obtain a general formula for such double-sums in terms of theta functions and false theta functions, which subsumes the decompositions of Matsusaka.  Our general formula is similar in structure to the case $b^2-ac>0$, where Mortenson and Zwegers obtain a decomposition in terms of Appell functions and theta functions.
\end{abstract}

\maketitle

\section{Introduction}
Let $q$ be a nonzero complex number with $|q|<1$.   We recall the $q$-Pochhammer notation:
\begin{equation*}
(x)_n=(x;q)_n:=\prod_{i=0}^{n-1}(1-q^ix), \ \ (x)_{\infty}=(x;q)_{\infty}:=\prod_{i\ge 0}(1-q^ix).
\end{equation*}
We also recall the basic definition of a theta function:
\begin{equation*}
\Theta(x;q):=(x)_{\infty}(q/x)_{\infty}(q)_{\infty}=\sum_{n=-\infty}^{\infty}(-1)^nq^{\binom{n}{2}}x^n,
\end{equation*}
where the last equality is the Jacobi triple product identity.    False theta functions are theta functions but with the wrong signs \cite{AW}.  Partial theta functions can be thought of as half of a theta function, in that we only sum over non-negative $n$.

Matsusaka \cite{Ma22} recently expressed false theta functions in terms of Hecke-type double-sums:
\begin{definition} \label{definition:fabc-def}  Let $x,y\in\mathbb{C} \backslash \{0\}$ and define $\sg (r):=1$ for $r\ge 0$ and $\sg(r):=-1$ for $r<0$. Then
\begin{equation}
f_{a,b,c}(x,y;q):=\sum_{r,s\in \mathbb{Z}}\sg(r,s)(-1)^{r+s}x^ry^sq^{a\binom{r}{2}+brs+c\binom{s}{2}}, \ \sg(r,s):=\left(\frac{\sg(r)+\sg(s)}{2}\right).\label{equation:fabc-def}\\
\end{equation}
\end{definition}
\noindent We note that we can also write
\begin{equation}
f_{a,b,c}(x,y;q)=\left( \sum_{r,s\ge 0 }-\sum_{r,s<0}\right)(-1)^{r+s}x^ry^sq^{a\binom{r}{2}+brs+c\binom{s}{2}}.
\label{equation:fabc-def2}
\end{equation}
The motivation for Matsusaka's work is a family of $q$-series originating from Habiro's work on Witten--Reshetikhin--Turaev invariants.  Matsusaka \cite{Ma22} generalizes the relation between five Habiro-type series \cite{Ha08} and false theta functions after Hikami \cite{Hi07}.  In \cite{Ma22}, five families of Habiro-type series are expressed in terms of sums of the form (\ref{equation:fabc-def2}) where $b^2-ac<0$.   As an example, one family reads \cite[Theorem 2.15]{Ma22}:
\begin{equation*}
H_{p}^{(2)}(q)=\frac{1}{(q)_{\infty}}\Big ( \sum_{a,b\ge 0}-\sum_{a,b<0}\Big ) 
(-1)^{a+b}q^{(p+1)a}q^{2b}q^{(2p+1)\binom{a}{2}+2ab+3\binom{b}{2}},
\end{equation*}
where $H_{p}^{(2)}(q)$ is a Habiro-type series, see Section \ref{section:examples}.  Matsusaka then decomposes the five families into sums of products of theta functions and false theta functions \cite[Theorems 3.12, 3.13, 3.15]{Ma22}.   The decompositions are then used to compute radial limits \cite{LZ}, \cite[Theorem 3.21]{Ma22}.

This setting contrasts with \cite{CK, Mo21B}, where false theta functions are expressed in terms of double-sums where instead there is a plus sign between the summation symbols in (\ref{equation:fabc-def2}) and we do not have the restriction $b^2-ac<0$.  It also contrasts with the setting where $b^2-ac>0$.

In \cite{HM, MZ}, double-sums of the form (\ref{equation:fabc-def2}) where $b^2-ac>0$ are extensively studied.  Expansions are obtained that express the double-sums in terms of Appell functions, i.e. the building blocks of Ramanujan's mock theta functions, and theta functions.   Here, Appell functions are defined
\begin{equation*}
m(x,z;q):=\frac{1}{\Theta(z;q)}\sum_{r=-\infty}^{\infty}\frac{(-1)^rq^{\binom{r}{2}}z^r}{1-q^{r-1}xz}.
\end{equation*}
The results in \cite{HM} are for general double-sums which enjoy certain symmetries.  For example, with one type of symmetry we have $b>a=c$ \cite[Theorem $1.3$]{HM}, of which a special case reads
\begin{align*}
f_{1,2,1}(x,y;q)
&=\Theta(y;q)m\Big (\frac{q^2x}{y^2},-1;q^3\Big )+\Theta(x;q)m\Big (\frac{q^2y}{x^2},-1;q^3\Big )\\
& \ \ \ \ \ \ \ \ \ \ - y\cdot \frac{(q^3;q^3)_{\infty}^3\Theta(-x/y;q)\Theta(q^2xy;q^3)}
{\Theta(-1;q^3)\Theta(-qy^2/x;q^3)\Theta(-qx^2/y;q^3)}.
\end{align*}
The above expansion has the immediate corollaries:
\begin{equation*}
f_{1,2,1}(q,q;q)=(q)_{\infty}^2, \ f_{1,2,1}(q,-q;q)=\Theta(-q;q^4)\phi(q),
\end{equation*}
where
\begin{equation*}
\phi(q):=\sum_{n=0}^{\infty}\frac{(-1)^nq^{n^2}(q;q^2)_{n}}{(-q)_{2n}}=2m(q,-1;q^3)
\end{equation*}
is a sixth-order mock theta function \cite{AH}.

In the course of resolving an open question on the modularity of certain duals of generalized quantum modular forms of Hikami and Lovejoy \cite{HL}, Mortenson and Zwegers \cite{MZ} obtained a decomposition for the general form (\ref{equation:fabc-def2}), where $b^2-ac>0$.  To state the decomposition, we first define the following expression involving Appell functions \cite{HM,Zw02}:
\begin{definition} Let $a,b,$ and $c$ be positive integers with  $D:=b^2-ac>0$.  Then
\begin{align} 
m_{a,b,c}(x,y,z_1,z_0;q)
&:=\sum_{t=0}^{a-1}(-y)^tq^{c\binom{t}{2}}\Theta(q^{bt}x;q^a)m\Big (-q^{a\binom{b+1}{2}-c\binom{a+1}{2}-tD}\frac{(-y)^a}{(-x)^b},z_0;q^{aD}\Big ) \label{equation:mabc-def}\\
&\ \ \ \ \ +\sum_{t=0}^{c-1}(-x)^tq^{a\binom{t}{2}}\Theta(q^{bt}y;q^c)m\Big (-q^{c\binom{b+1}{2}-a\binom{c+1}{2}-tD}\frac{(-x)^c}{(-y)^b},z_1;q^{cD}\Big ).\notag
\end{align}
\end{definition}
\begin{theorem}\cite[Corollary 4.2]{MZ} Let $a,b,$ and $c$ be positive integers with  $D:=b^2-ac>0$. For generic $x$ and $y$, we have
\begin{align*}
& f_{a,b,c}(x,y;q)=m_{a,b,c}(x,y,-1,-1;q)+\frac{1}{\Theta(-1;q^{aD})\Theta(-1;q^{cD})}\cdot \theta_{a,b,c}(x,y;q),
\end{align*}
where
\begin{align*}
&\theta_{a,b,c}(x,y;q):=
\sum_{d^*=0}^{b-1}\sum_{e^*=0}^{b-1}q^{a\binom{d-c/2}{2}+b( d-c/2 ) (e+a/2  )+c\binom{e+a/2}{2}}(-x)^{d-c/2}(-y)^{e+a/2}\\
&\cdot\sum_{f=0}^{b-1}q^{ab^2\binom{f}{2}+\big (a(bd+b^2+ce)-ac(b+1)/2 \big )f} (-y)^{af}
\cdot \Theta(-q^{c\big ( ad+be+a(b-1)/2+abf \big )}(-x)^{c};q^{cb^2})\\
&\cdot \Theta(-q^{a\big ( (d+b(b+1)/2+bf)(b^2-ac) +c(a-b)/2\big )}(-x)^{-ac}(-y)^{ab};q^{ab^2D})\\
&\cdot \frac{(q^{bD};q^{bD})_{\infty}^3\Theta(q^{ D(d+e)+ac-b(a+c)/2}(-x)^{b-c}(-y)^{b-a};q^{bD})}
{\Theta(q^{De+a(c-b)/2}(-x)^b(-y)^{-a};q^{bD})\Theta(q^{Dd+c(a-b)/2}(-y)^b(-x)^{-c};q^{bD})}.
\end{align*}
Here $d:=d^*+\{c/2 \}$ and $e:=e^*+\{ a/2\}$, with  $0\le \{\alpha \}<1$ denoting fractional part of $\alpha$.
\end{theorem}

When $b^2-ac>0$, one can obtain an Appell function expression such as (\ref{equation:mabc-def}) by first determining the appropriate functional equation for (\ref{equation:fabc-def}) and iterating it.  If one starts to see divergent partial theta functions, one uses a heuristic that relates divergent partial theta functions and Appell functions in order to express (\ref{equation:mabc-def}) in terms of Appell functions up to a theta function.  See for example \cite[Section 3]{HM}, \cite[Section 4.1]{Mo14}, \cite[Section 4]{Mo21A}, \cite[Section 8]{Mo17}.  Determining the theta function is a difficult task.  Sometimes, one can obtain the theta function in the course of a direct proof \cite{Mo9, MZ}.

When $b^2-ac<0$, iterating the appropriate functional equation for (\ref{equation:mabc-def}) yields partial theta functions that do not diverge.  This makes the situation straight forward and also leads us to our main result.  
\begin{theorem}\label{theorem:mainTheorem}
Let $a,b,$ and $c$ be positive integers with  $D:=b^2-ac<0$. For generic $x$ and $y$, we have
 that
\begin{align} 
f_{a,b,c}(x,y;q)
&=\frac{1}{2}\Big (  \sum_{t=0}^{a-1}(-y)^tq^{c\binom{t}{2}}\Theta(q^{bt}x;q^a)
\sum_{r\in\mathbb{Z}}\sg(r) \left (q^{a\binom{b+1}{2}-c\binom{a+1}{2}-tD}\frac{(-y)^a}{(-x)^b}\right )^r
q^{-aD\binom{r+1}{2}}\label{equation:mainIdentity} \\
&\ \ \ \ \ +\sum_{t=0}^{c-1}(-x)^tq^{a\binom{t}{2}}\Theta(q^{bt}y;q^c)
\sum_{r\in\mathbb{Z}}\sg(r)\left (q^{c\binom{b+1}{2}-a\binom{c+1}{2}-tD}\frac{(-x)^c}{(-y)^b}\right )^rq^{-cD\binom{r+1}{2}}\Big ) .\notag
\end{align}
\end{theorem}
\noindent Matsusaka \cite[Theorems 3.12, 3.13, 3.15]{Ma22} only computes decompositions for the case $(a,b,c)=(2p+1,2,3)$ and certain values of $x$ and $y$.  Our Theorem \ref{theorem:mainTheorem} is for $(a,b,c)$ with $b^2-ac<0$.  For another example where iterating the functional equation yields convergent partial theta functions, see \cite[Section 7]{Mo21A}, which discusses triple-sum partial theta function identities found in \cite{KL}.

In Section \ref{section:examples}, we discuss examples related to results in \cite{Hi07, Ma22}.  In Section \ref{section:proof} we prove Theorem \ref{theorem:mainTheorem}.

\section{Examples}\label{section:examples}
In \cite{Ma22}, Matsusaka generalized Hikami's examples of Habiro-type series \cite{Hi07} to five infinite families that are of $q$-hypergeometric multi-sum form:
{\allowdisplaybreaks \begin{align*}
H_{p}^{(1)}(q)&:=\sum_{s_p\ge \dots \ge s_1\ge0}q^{s_p}(q^{s_p+1})_{s_p+1}\prod_{i=1}^{p-1}q^{s_i(s_i+1)}
\left [ \begin{matrix} s_{i+1}\\s_i\end{matrix}\right ]_{q},\\
H_{p}^{(2)}(q)&:=\sum_{s_p\ge \dots \ge s_1\ge0}q^{s_p}(q^{s_p})_{s_p+1}\prod_{i=1}^{p-1}q^{s_{i}^2}
\left [ \begin{matrix} s_{i+1}\\s_i\end{matrix}\right ]_{q},\\
H_{p}^{(3)}(q)&:=\sum_{s_p\ge \dots \ge s_1\ge0}q^{2s_p}(q^{s_p+1})_{s_p+1}\prod_{i=1}^{p-1}q^{s_i(s_i+1)}
\left [ \begin{matrix} s_{i+1}\\s_i\end{matrix}\right ]_{q},\\
H_{p}^{(4)}(q)&:=\sum_{s_p\ge \dots \ge s_1\ge0}q^{s_p}(q^{s_p+1})_{s_p}\prod_{i=1}^{p-1}q^{s_i(s_i+1)}
\left [ \begin{matrix} s_{i+1}\\s_i\end{matrix}\right ]_{q},\\
H_{p}^{(5)}(q)&:=\sum_{s_p\ge \dots \ge s_1\ge0}q^{s_p}(q^{s_p+1})_{s_p+1}\prod_{i=1}^{p-1}q^{s_{i}^2}
\left [ \begin{matrix} s_{i+1}\\s_i\end{matrix}\right ]_{q}.
\end{align*}}%
Matsusaka then employs Bailey's Lemma to express each family in terms of the Hecke-type double-sum form (\ref{equation:fabc-def2}) where $b^2-ac<0$.  Hikami \cite{Hi07} had already obtained the first and fifth series.  Specifically, Hikami and Matsusaka obtain the following Hecke-type expansions of the generalized Habiro-type series \cite{Hi07}, \cite[Section 2.3]{Ma22}:
{\allowdisplaybreaks \begin{align*}
H_{p}^{(1)}(q)&=\frac{1}{(q)_{\infty}}\Big ( \sum_{a,b\ge 0}-\sum_{a,b<0}\Big ) 
(-1)^{a+b}q^{(2p+1)a}q^{4b}q^{(2p+1)\binom{a}{2}+2ab+3\binom{b}{2}},\\
H_{p}^{(2)}(q)&=\frac{1}{(q)_{\infty}}\Big ( \sum_{a,b\ge 0}-\sum_{a,b<0}\Big ) 
(-1)^{a+b}q^{(p+1)a}q^{2b}q^{(2p+1)\binom{a}{2}+2ab+3\binom{b}{2}},\\
H_{p}^{(3)}(q)&=\frac{1}{(q)_{\infty}}\Big ( \sum_{a,b\ge 0}-\sum_{a,b<0}\Big ) 
(-1)^{a+b}q^{(2p+2)a}q^{3b}q^{(2p+1)\binom{a}{2}+2ab+3\binom{b}{2}},\\
H_{p}^{(4)}(q)&=\frac{1}{(q)_{\infty}}\Big ( \sum_{a,b\ge 0}-\sum_{a,b<0}\Big ) 
(-1)^{a+b}q^{(2p+1)a}q^{2b}q^{(2p+1)\binom{a}{2}+2ab+3\binom{b}{2}},\\
H_{p}^{(5)}(q)&=\frac{1}{(q)_{\infty}}\Big ( \sum_{a,b\ge 0}-\sum_{a,b<0}\Big ) 
(-1)^{a+b}q^{(p+1)a}q^{3b}q^{(2p+1)\binom{a}{2}+2ab+3\binom{b}{2}}.
\end{align*}}%
Matsusaki then computes the false theta function decompositions for the above five families in \cite[Theorems 3.12, 3.13, 3.15]{Ma22}.  In \cite[Theorem 3.21]{Ma22}, he computes the radial limits.

We demonstrate our results by computing a few examples, but first remind the reader of some easily derived properties for theta functions.  We recall that
\begin{equation*}
\Theta(q/x;q)=\Theta(x;q) \textup{ and } \Theta(q^2;q^3)=\Theta(q;q^3)=(q)_{\infty}.
\end{equation*}
When $n\in\mathbb{Z}$ we have
$\Theta(q^n;q)=0$ 
and the elliptic transformation property:
\begin{equation}
\Theta(q^{n}x;q)=(-1)^{n}x^{-n}q^{-\binom{n}{2}}\Theta(x;q).\label{equation:j-elliptic}
\end{equation}

For $H_{1}^{(2)}(q)$, our Theorem \ref{theorem:mainTheorem} yields
{\allowdisplaybreaks \begin{align*} 
H_{1}^{(2)}(q)&=\frac{1}{(q)_{\infty}}f_{3,2,3}(q^2,q^2;q)\\
&=\frac{1}{(q)_{\infty}} \sum_{t=0}^{2}(-q^2)^tq^{3\binom{t}{2}}\Theta(q^{2t+2};q^3)
\sum_{r\in\mathbb{Z}}\sg(r)(-q^{-7+5t})^rq^{15\binom{r+1}{2}} \\
&=\frac{1}{(q)_{\infty}}\Big (  \Theta(q^{2};q^3)
\sum_{r\in\mathbb{Z}}\sg(r)(-1)^rq^{-7r}q^{15\binom{r+1}{2}}\\
&\ \ \ \ \ - q^2\Theta(q^{4};q^3)
\sum_{r\in\mathbb{Z}}\sg(r)(-1)^rq^{-2r}q^{15\binom{r+1}{2}}\\
&\ \ \ \ \ + q^{7}\Theta(q^{6};q^3)
\sum_{r\in\mathbb{Z}}\sg(r)(-q^{3})^rq^{15\binom{r+1}{2}}\Big  )\\
&=
\sum_{r\in\mathbb{Z}}\sg(r)(-1)^rq^{8r}q^{15\binom{r}{2}}
 + q\sum_{r\in\mathbb{Z}}\sg(r)(-1)^rq^{13r}q^{15\binom{r}{2}}.
\end{align*}}%

For $H_{2}^{(2)}(q)$, our Theorem \ref{theorem:mainTheorem} yields
{\allowdisplaybreaks \begin{align*}
H_{2}^{(2)}(q)&=\frac{1}{(q)_{\infty}}f_{5,2,3}(q^3,q^2;q)\\
&=\frac{1}{2(q)_{\infty}}\Big (  \Theta(q^{3};q^5)
\sum_{r\in\mathbb{Z}}\sg(r)(-1)^r q^{-26r}q^{55\binom{r+1}{2}} 
 -q^{5}\Theta(q^{2};q^5)
\sum_{r\in\mathbb{Z}}\sg(r)(-1)^r q^{-4r}q^{55\binom{r+1}{2}} \\
&\ \ \ \ \ +q^{11}\Theta(q^{4};q^5)
\sum_{r\in\mathbb{Z}}\sg(r)(-1)^r q^{7r}q^{55\binom{r+1}{2}}
+  q^{19}\Theta(q;q^5)
\sum_{r\in\mathbb{Z}}\sg(r)(-1)^r q^{18r}q^{55\binom{r+1}{2}} \\
&\ \ \ \ \ +\Theta(q^{2};q^3)
\sum_{r\in\mathbb{Z}}\sg(r)(-1)^rq^{-16r}q^{33\binom{r+1}{2}} 
 +q^2\Theta(q;q^3)
\sum_{r\in\mathbb{Z}}\sg(r)(-1)^rq^{-5r}q^{33\binom{r+1}{2}}\Big ).
\end{align*}}%

\section{Proof of Theorem \ref{theorem:mainTheorem}} \label{section:proof}

For reference, we recall that
\begin{align*}
f_{a,b,c}(x,y,q)
&=\sum_{m\in\mathbb{Z}} \sum_{n\in\mathbb{Z}} 
\left(\frac{\sg(m)+\sg(n)}{2}\right)(-1)^{r+s}x^my^nq^{a\binom{m}{2}+bmn+c\binom{n}{2}}.
\end{align*}
Using the (\ref{equation:j-elliptic}), we have
\begin{equation}
\Theta(q^{b(ar+t)}x;q^a)=\Theta(q^{abr}xq^{bt};q^a)
=(-x)^{-br}q^{-tb^2r}q^{-a\binom{br}{2}}\Theta(q^{bt}x;q^a). \label{equation:helpingIdentity}
\end{equation}
\begin{proof}[Proof of Theorem \ref{theorem:mainTheorem}] We want to demonstrate that the right-hand side of (\ref{equation:mainIdentity}) equals the left-hand side of  (\ref{equation:mainIdentity}).  We consider the first summand on the right-hand side.  We have
{\allowdisplaybreaks \begin{align*}
 \sum_{t=0}^{a-1}&(-y)^tq^{c\binom{t}{2}}\Theta(q^{bt}x;q^a)
\sum_{r\in\mathbb{Z}}\sg(r) \left (q^{a\binom{b+1}{2}-c\binom{a+1}{2}-tD}\frac{(-y)^a}{(-x)^b}\right )^r
q^{-aD\binom{r+1}{2}}\\
&= \sum_{t=0}^{a-1}q^{c\binom{t}{2}}\Theta(q^{bt}x;q^a)
\sum_{r\in\mathbb{Z}}\sg(ar+t) \left (q^{a\binom{b+1}{2}-c\binom{a+1}{2}-tD}\right )^r\frac{(-y)^{ar+t}}{(-x)^{br}}
q^{-aD\binom{r+1}{2}}\\
&= \sum_{t=0}^{a-1}q^{c\binom{t}{2}}\sum_{r\in\mathbb{Z}}\sg(ar+t) \Theta(q^{b(ar+t)}x;q^a)(-x)^{br}q^{tb^2r}
q^{a\binom{br}{2}}\\
&\qquad \cdot\left (q^{a\binom{b+1}{2}-c\binom{a+1}{2}-tD}\right )^r\frac{(-y)^{ar+t}}{(-x)^{br}}
q^{-aD\binom{r+1}{2}}\\
&= \sum_{t=0}^{a-1}\sum_{r\in\mathbb{Z}}\sg(ar+t) \Theta(q^{b(ar+t)}x;q^a)q^{c\binom{ar+t}{2}}(-y)^{ar+t}\\
&= \sum_{n\in\mathbb{Z}}\sg(n) \Theta(q^{bn}x;q^a)q^{c\binom{n}{2}}(-y)^{n}\\
&= \sum_{n\in\mathbb{Z}}\sg(n) q^{c\binom{n}{2}}(-y)^{n}\sum_{m\in\mathbb{Z}}(-x)^{m}q^{bmn}q^{a\binom{m}{2}},
\end{align*}}%
where we used (\ref{equation:helpingIdentity}), substituted $n=ar+t$, and then used the Jacobi triple product identity.  The argument for the second summand on the right-hand side of (\ref{equation:mainIdentity}) is analogous.\qedhere
\end{proof}

\section*{Acknowledgements}
tbd.

\end{document}